\documentclass[12pt]{amsart}

\usepackage[utf8]{inputenc}
\usepackage{amsthm}
\usepackage{amssymb}
\usepackage{amsfonts}
\usepackage{cite}
\usepackage{enumerate}
\usepackage{fullpage}
\usepackage{hyperref}
\usepackage{bbm}
\usepackage{color}

\author{Luiz Moreira}

\title{Ramsey Goodness of Paths in Random Graphs}
\address{IMPA, Estrada Dona Castorina 110, Jardim Bot\^anico, Rio de Janeiro, RJ, Brasil}
\email{luizfm@impa.br}
\thanks{This research was partially supported by CNPq}

\newtheorem{thm}{Theorem}[section]

\newtheorem{lemma}[thm]{Lemma}

\newtheorem{claim}[thm]{Claim}
\newtheorem{prop}[thm]{Proposition}
\newtheorem{prob}[thm]{Problem}

\def\le{\leqslant}
\def\leq{\leqslant}
\def\ge{\geqslant}
\def\geq{\geqslant}

\def\E{\mathbb{E}}

\def\P{\mathbb{P}}
\def\N{\mathbb{N}}
\def\eps{\varepsilon}

\begin{document}

\begin{abstract}
We say that a graph $G$ is \emph{Ramsey} for $H_1$ versus $H_2$, and write $G \to (H_1,H_2)$, if every red-blue colouring of the edges of $G$ contains either a red copy of $H_1$ or a blue copy of $H_2$. In this paper we study the threshold for the event that the Erd\H{o}s--R\'enyi random graph $G(N,p)$ is Ramsey for a clique versus a path. We show that
$$G\big( (1 + \eps) rn,p \big) \to (K_{r+1},P_n)$$
with high probability if $p \gg n^{-2 / (r + 1)}$, and
$$G\big( rn + t, p \big) \to (K_{r+1},P_n)$$
with high probability if $p \gg n^{-2 / (r + 2)}$ and $t \gg 1/p$. Both of these results are sharp (in different ways), since with high probability $G(Cn,p) \not\to (K_{r+1}, P_n)$ for any constant $C > 0$ if $p \ll n^{-2/(r + 1)}$, and $G(rn + t, p) \not\to (K_{r+1}, P_n)$ if $t \ll 1/p$, for any $0 < p \le 1$. 
\end{abstract}

\maketitle

\section{Introduction}

Let us write $G \to (H_1,H_2)$ if every red-blue colouring of the edges of $G$ contains either a red copy of $H_1$ or a blue copy of $H_2$. Ramsey~\cite{R30} proved in 1930 that, for every pair of graphs $(H_1,H_2)$, there exists a graph $G$ such that $G \to (H_1,H_2)$, and the \emph{Ramsey number} $R(H_1,H_2)$ is defined to be the minimum number of vertices of such a graph. In the decades since, Graph Ramsey Theory has developed into a rich and deep area of study, containing many beautiful theorems and powerful techniques, see for example the survey~\cite{CFS}.

In this paper we will be interested in the Ramsey properties of random graphs, the study of which was initiated by Frankl and R\"odl~\cite{FR86} and by Luczak, Ruci\'nski, and Voigt~\cite{LRV}, who proved that $p = n^{-1/2}$ is a threshold for the event that $G(n,p) \to K_3$ (where we write $G \to H$ as a shorthand for $G \to (H,H)$). The threshold for an arbitrary fixed graph $H$ was determined by R\"odl and Ruci\'nski~\cite{RR95} in 1995, who proved that (except in a few simple special cases),
\begin{equation*}\label{eq:rmsfix}
\lim_{n \to \infty} \P\big( G(n,p) \to H \big) =
\begin{cases}
1 & \text{if $p \gg n^{-1/m_2(H)}$} \\
0 & \text{if $p \ll n^{-1/m_2(H)}$},
\end{cases}
\end{equation*}
\noindent where $m_2(H) = \max \big\{ \frac{e(F) - 1}{v(F) - 2} : F \subset H \text{ with } v(F) \geq 3 \big\}$. (The authors of~\cite{RR95} also proved that the same result holds for an arbitrary (fixed) number of colours.) The problem for pairs $(H_1,H_2)$ with $H_1 \ne H_2$ is more difficult, and a conjecture of Kreuter and Kohayakawa~\cite{KK} on the location of the threshold has remained open for over 20 years. In a recent breakthrough, however, the $1$-statement conjectured in~\cite{KK} was proved by Mousset, Nenadov, and Samotij~\cite{MNS}, using the method of hypergraph containers (see~\cite{BMS,ST}, or the survey~\cite{BMS18}).


Another area of Ramsey theory in which random graphs have played an important role is the study of the \emph{size Ramsey number}
$$\hat{R}(H) := \min\big\{ e(G) : G \to H \big\}$$
of a graph $H$. In particular, Beck~\cite{B83} used a sparse random graph to show that that
$$\hat{R}(P_n) = O(n),$$
where $P_n$ is the path with $n$ edges, disproving a conjecture of Erd\H{o}s~\cite{E81}. More precisely, Beck proved that there exists a constant $C > 0$ such that
$$G\big(Cn, C/n \big) \to P_n$$
with high probability as $n \to \infty$. We remark that an asymptotically optimal variant of this result was obtained recently by Letzter~\cite{L16}, who proved that if $\eps > 0$ and $pn \to \infty$, then
$$G\big((3/2 + \eps)n, p\big) \to P_n$$
with high probability. For some recent generalisations of Beck's result, see~\cite{BKMMMMP,CJKMMRR,HJKMR,KLWY}.

We will be interested in a third direction of research in Ramsey theory, whose systematic study was initiated by Burr and Erd\H{o}s~\cite{BE83} in 1983. The authors of~\cite{BE83} were inspired by a result of Chv\'atal~\cite{C77}, which states that
\begin{equation}\label{eq:Chvatal}
R(K_r,T) = \big( r - 1 \big)\big( |T| - 1 \big) + 1
\end{equation}
for every $r \in \N$, and every tree $T$. Burr~\cite{BE83} observed that the construction used to prove the lower bound in~\eqref{eq:Chvatal} can be modified to show that, if $G$ is connected and $|G| \ge \sigma(H)$, then
\begin{equation}\label{eq:Ramsey:good}
R(H,G) \ge \big( \chi(H) - 1 \big)\big( |G| - 1 \big) + \sigma(H)
\end{equation}
where $\sigma(H)$ is the minimum size of a colour class in a proper $\chi(H)$-colouring of $H$. Indeed, to prove~\eqref{eq:Ramsey:good}, consider the colouring consisting of $\chi(H)-1$ disjoint red cliques of size $|G|-1$, and one additional disjoint red clique of size $\sigma(H)-1$.

Following Burr and Erd\H{o}s~\cite{BE83}, we say that a graph $G$ is \emph{Ramsey $H$-good} (or just \emph{$H$-good}) if equality holds in~\eqref{eq:Ramsey:good}. Note that every tree is $K_r$-good for every $r \in \N$, by~\eqref{eq:Chvatal}, but it turns out that there exist pairs $(H,T)$, with $T$ a tree, such that $T$ is not $H$-good.
Nevertheless, it was proved by Erd\H{o}s, Faudree, Rousseau and Schelp~\cite{EFRS} that the path $P_n$ is $H$-good for every fixed $H$ and all sufficiently large $n \in \N$, and their result was recently strengthened by Pokrovskiy and Sudakov~\cite{PS}, who proved that $P_n$ is $H$-good for all $n \ge 4|H|$. For work on some of the many other questions and conjectures posed by Burr and Erd\H{o}s~\cite{BE83} about the family of $H$-good graphs, see for example~\cite{ABS,FGMSS,NR}, and the references therein.

In this paper we initiate the study of Ramsey-goodness in sparse random graphs. In particular, for each $r \in \N$ we will give bounds on the pairs $(N,p)$, where $N  = N(n) \in \N$ and $p = p(n) \in (0,1)$, such that
\begin{equation}\label{eq:our:event}
G(N,p) \to (K_{r+1},P_n)
\end{equation}
with high probability. Our main results are as follows: the first determines the threshold for the event~\eqref{eq:our:event} when $N = (1 + \eps)R(K_{r+1},P_n)$ for some fixed $\eps > 0$.

\begin{thm}\label{thm:general}
Let $2 \le r \in \N$, and let $p \gg n^{-2/(r + 1)}$. For every fixed $\eps > 0$, we have
$$G\big( (1 + \eps) rn, p \big) \to \big( K_{r + 1}, P_n \big)$$
with high probability as $n \to \infty$.
\end{thm}

The lower bound on $p$ is necessary, since we will show (see Proposition~\ref{prop:naturalbarrier}, below) that if $p \ll n^{-2/(r + 1)}$, then $G(Cn,p) \not\to (K_{r + 1}, P_n)$ with high probability for every constant $C > 0$. We remark that we will actually prove a somewhat stronger result than that stated above when $p$ is significantly larger than $n^{-2/(r + 1)}$, see Theorem~\ref{thm:general:precise}, below.

It is natural to ask for the smallest $t$ such that $G(rn + t,p) \to \big( K_{r + 1}, P_n \big)$ with high probability.
Our second main result resolves this problem up to a constant factor, but only for slightly larger values of $p$.

\begin{thm}\label{thm:klregime}
For each $2 \le r \in \N$, there exists a constant $C = C(r) > 0$ such that the following holds. If $p \gg n^{-2/(r + 2)}$ and $t \ge C/p$, then
$$G\big( rn + t, p \big) \to \big( K_{r + 1}, P_n \big)$$
with high probability as $n \to \infty$.
\end{thm}

The lower bound on $t$ in Theorem~\ref{thm:klregime} is best possible up to the value of the constant $C$, since we will show (see Proposition~\ref{allpconstruction}, below) that, for every $r \in \N$ and any function $p = p(n)$, if $tp \to 0$ as $n \to \infty$ then $G(rn + t, p) \not\to (K_{r + 1}, P_n)$ with high probability. Moreover, the lower bound on $p$ in the theorem is also best possible, since we will show that if $p \ll n^{-2/(r + 2)}$ and $tp \to \infty$ slowly, then $G(rn + t, p) \not\to (K_{r + 1}, P_n)$ with high probability (see Proposition~\ref{prop:specific}, below, for the precise statement).

Let us briefly sketch the proofs of Theorems~\ref{thm:general} and~\ref{thm:klregime}. The proof of Theorem~\ref{thm:general} is relatively simple, the main step being Proposition~\ref{prop:bluepartite}, which says that any red-blue colouring of a graph $G$ on $rn + (r+1)t$ vertices contains either a red $P_n$, or $r + 1$ sets of size $t$ with no red edges between them. We prove this by repeatedly applying a lemma of Krivelevich and Sudakov~\cite{KS}. To deduce the theorem from this proposition, we apply Janson's inequality.

The proof of Theorem~\ref{thm:klregime} is somewhat more complicated, but our effort will be rewarded with a stronger `structural stability' result (see Theorem~\ref{thm:structural}). Let $N \ge rn - n/3$, and let $(G_R,G_B)$ be a red-blue colouring of $G(N,p)$ containing neither a red $P_n$ nor a blue $K_{r + 1}$. We first apply the sparse regularity lemma and the so-called K\L R Conjecture (recently proved in~\cite{BMS,CGSS,ST}), to deduce (see Lemma~\ref{lem:structural}) that $G_R$ contains $r$ vertex-disjoint red cycles, each of length greater than $n/2$. Next, we observe that if a vertex sends red edges to more than one of these cycles then $P_n \subset G_R$, and show (using Janson's inequality) that if any vertex sends $\Omega(pn)$ blue edges into each of the cycles then $K_{r+1} \subset G_B$. We can then partition the vertex set according to the red cycle to which a vertex sends $\Omega(pn)$ red edges, with a `trash' set of size $O(1/p)$ consisting of vertices that send only $o(pn)$ edges to some cycle. Finally, we prove that the red edges inside each part satisfy a simple expansion property, and deduce (using a lemma of P\'osa~\cite{P76}) that each contains a red Hamilton path. It follows that, with high probability, there exists an almost-cover of the vertex set (missing only $O(1/p)$ vertices) by $r$ sets of size at most $n$, with no red edges between them, and this implies Theorem~\ref{thm:klregime}.

The rest of the paper is organised as follows: in Section~\ref{sec:constructions} we describe constructions that imply the 0-statements, in Section~\ref{sec:gen1st} we prove Theorem~\ref{thm:general}, and in Section~\ref{sec:klr1st} we prove Theorem~\ref{thm:klregime}. We conclude, in Section~\ref{sec:openprob}, by stating some open problems.

\section{Constructions: proof of the 0-statements}\label{sec:constructions}

In this section we describe some (very simple) colourings that demonstrate the sharpness of the results stated in the Introduction. We begin with the following proposition, which shows that the bound on $p$ in Theorem~\ref{thm:general} is necessary.

\begin{prop}\label{prop:naturalbarrier}
Let $2 \le r \in \N$, and let $p \ll n^{-2/(r + 1)}$. For any constant $C > 0$,
$$G\big( Cn, p \big) \not\to \big( K_{r + 1}, P_n \big)$$
with high probability as $n \to \infty$.
\end{prop}

\begin{proof}
Define the random variable $X$ to be the number of copies of $K_{r+1}$ in $G(Cn,p)$, and observe that
$$\E[X] = \binom{Cn}{r + 1} p^{\binom{r + 1}{2}} = O\Big( n^{r + 1}p^{\binom{r + 1}{2}} \Big).$$
By Markov's inequality, it follows that
$$\P\big( X \geq n \big) \le \dfrac{\E[X]}{n} = O\Big( n^r p^{\binom{r + 1}{2}} \Big) \ll 1.$$
However, if $X < n$ then by colouring one edge of each copy of $K_{r+1}$ red, and all other edges blue, we can guarantee there will be no blue $K_{r+1}$ and no red $P_n$ (since at most $n - 1$ edges are red, and $P_n$ has $n$ edges).
\end{proof}

The next proposition shows that the bound on $t$ in Theorem~\ref{thm:klregime} is best possible.

\begin{prop}\label{allpconstruction}
Let $2 \le r \in \N$, and let $p = p(n) \in (0,1)$. If $t \ll 1/p$, then
$$G\big( rn + t, p \big) \not\to \big( K_{r + 1}, P_n \big)$$
with high probability as $n \to \infty$.
\end{prop}

\begin{proof}
Let $G = G(rn + t, p)$. Our plan is to find a partition $V(G) = A_0 \cup A_1 \cup \cdots \cup A_r$ of the vertex set into $r+1$ parts, with $|A_0| = t$ and $|A_1| = \cdots = |A_r| = n$, and colour edges red if they are inside one of the parts, and blue otherwise. Such a colouring clearly contains no red copy of $P_n$ (since $P_n$ has $n+1$ vertices and is connected), and if there exist two parts with no edges between them, then it will have no blue copy of $K_{r+1}$. In fact, we will find a partition such that for every edge leaving $A_0$, the other endpoint is in $A_1$.

To do so, let $A_0$ be an arbitrary set of $t$ vertices, and define a random variable
$$X := \bigcup_{v \in A_0} N_G(v) \setminus A_0,$$
so $X$ is the union of the neighbourhoods (outside $A_0$) of the vertices in $A_0$. Now,
$$\E\big[ |X| \big] = \sum_{u \not\in A_0} \P\big( N_G(u) \cap A_0 \ne \emptyset \big) \le \sum_{u \not\in A_0} \sum_{v \in A_0} \P\big( uv \in E(G) \big) = rntp \ll n,$$
so, by Markov's inequality, $X < n$ with high probability.

Now, if $X < n$ then we may choose a partition $V(G) \setminus A_0 = A_1 \cup \cdots \cup A_r$ with $|A_1| = \cdots = |A_r| = n$ and $X \subset A_1$. By the comments above, the red-blue colouring given by such a partition contains no red copy of $P_n$, and no blue copy of $K_{r+1}$.
 \end{proof}

To finish the section, let us record the following bound, which is stronger than that given by Proposition~\ref{prop:naturalbarrier}, and also stronger than that given by Proposition~\ref{allpconstruction} if $p \ll n^{-2/(r + 2)}$.

\begin{prop}\label{prop:specific}
Let $2 \le r \in \N$, and let $p = p(n) \in (0,1)$. If $t \ll p^{-\binom{r + 1}{2}}n^{-(r - 1)}$, then
$$G\big( rn + t, p) \not\to (K_{r+1}, P_n)$$
with high probability as $n \to \infty$.
\end{prop}

\begin{proof}
Our plan is again to choose a partition $V\big( G(rn + t,p) \big) = A_0 \cup A_1 \cup \cdots \cup A_r$ with $|A_0| = t$ and $|A_1| = \cdots = |A_r| = n$, and colour edges red if they are inside one of the parts, and blue otherwise. However, this time we will choose the partition so that for every copy of $K_{r+1}$ with exactly one vertex in $A_0$, the other vertices are all in $A_1$.

To do so, let $A_0$ be an arbitrary set of $t$ vertices, and define the random variable $X$ to be the number of copies of $K_{r+1}$ in $G = G(rn + t,p)$ with exactly one vertex in $A_0$. Observe that
$$\E[X] = t {rn \choose r} p^{{r + 1 \choose 2}} = O\Big( t \cdot n^r p^{{r + 1 \choose 2}} \Big) \ll n.$$
so, by Markov's inequality, $X < n/r$ with high probability.

Now, if $X < n/r$, then there exists a partition $V(G) \setminus A_0 = A_1 \cup \cdots \cup A_r$, with $|A_1| = \cdots = |A_r| = n$, such that every copy of $K_{r+1}$ in $G$ with exactly one vertex in $A_0$ is contained in $A_0 \cup A_1$. Choose such a partition, and colour edges red if they are inside one of the parts, and blue otherwise. This colouring contains no red copy of $P_n$ (as before), and no blue copy of $K_{r+1}$, since any such subgraph must have exactly one vertex in each part.
\end{proof}

\section{A general 1-statement: the proof of Theorem~\ref{thm:general}}
\label{sec:gen1st}

In this section we will prove Theorem~\ref{thm:general}, which (together with Proposition~\ref{prop:naturalbarrier}) determines the threshold for the (asymptotic) Ramsey goodness of the path in a random graph. We will in fact prove the following strengthening of the result stated in the Introduction.

\begin{thm}\label{thm:general:precise}
Let $2 \le r \in \N$, and let $p = x n^{-2/(r + 1)}$, with $x \gg 1$. If $t \ge p^{-(r + 1)/2} \log x$, then
$$G\big( rn + t, p \big) \to \big( K_{r + 1}, P_n \big)$$
with high probability as $n \to \infty$.
\end{thm}

We will deduce the theorem from the following deterministic proposition. To simplify the statement, let us say that a red-blue colouring of a graph $G$ \emph{weakly contains a blue $K_{r+1}(t)$} if there exist $r+1$ disjoint sets $A_1,\ldots,A_{r+1} \subset V(G)$ with $|A_1| = \cdots = |A_{r+1}| = t$ such that for each $1 \le i < j \le r+1$, every edge of $G$ between $A_i$ and $A_j$ is coloured blue.

\begin{prop}\label{prop:bluepartite}
Let $2 \le r \in \N$, let $t \in \N$, and let $G$ be a graph on $rn + (r+1)t$ vertices. Every red-blue colouring of $E(G)$ either contains a red $P_n$, or weakly contains a blue $K_{r+1}(t)$.
\end{prop}

The key tool we will need to prove Proposition~\ref{prop:bluepartite}, is the following lemma\footnote{To be precise, Lemma~\ref{lemma:bigpath} is an asymmetric version of that given in~\cite{BKS,K,KS}, but the proof is identical.} of Ben-Eliezer, Krivelevich and Sudakov~\cite{BKS} (see also~\cite{KS,K}).

\begin{lemma}\label{lemma:bigpath}
Let $k < n$ be positive integers, and let $G$ be a graph on $n$ vertices containing no path of length $n - k$. For each pair $(a,b)$ of positive integers with $a + b = k$, there exist disjoint sets $A,B \subset V(G)$, with $|A| = a$ and $|B| = b$, such that $e(A,B) = 0$.
\end{lemma}

We are now ready to prove Proposition~\ref{prop:bluepartite}.

\begin{proof}[Proof of Proposition \ref{prop:bluepartite}]
We will repeatedly apply Lemma~\ref{lemma:bigpath} in order to find the sets $A_i$ one by one. Let $G$ be a graph on $rn + (r+1)t$ vertices, and let $c$ be a colouring of the edges of $G$ containing no red $P_n$. The following claim is designed to facilitate a proof by induction.

\begin{claim}\label{claim:4}
For each $1 \le s \le r$, there exist disjoint sets $A_1,\ldots,A_{s+1} \subset V(G)$ with
$$|A_1| = \cdots = |A_s| = t \qquad \text{and} \qquad |A_{s+1}| = (r - s)n + (r - s + 1)t,$$
such that for each $1 \le i < j \le s + 1$, every edge of $G$ between $A_i$ and $A_j$ is coloured blue.
\end{claim}

\begin{proof}[Proof of Claim \ref{claim:4}]
The proof is by induction on $s$. For the base case, $s = 1$, we wish to find disjoint sets $A_1, A_2 \subset V(G)$ with $|A_1| =  t$ and $|A_2| = (r - 1)n + rt$, and with all edges of $G[A_1,A_2]$ blue. Applying Lemma~\ref{lemma:bigpath} to the red graph, and recalling that $G$ contains no path of length
$$rn + (r+1)t - t - \big( (r - 1)n + rt \big) = n,$$
it follows that such sets must exist.

Now, suppose that the claim is true for $s - 1$, and let $A_1,\ldots,A_s$ be the sets given by the induction hypothesis. We now simply repeat the argument above inside $A_s$. To be precise, if there is at least one red edge between every pair of subsets $X,Y \subset A_s$ with $|X| = t$ and $|Y| = (r - s)n + (r - s + 1)t$, then applying Lemma~\ref{lemma:bigpath} to the red graph on $A_s$, it follows that the red graph contains a path of length at least
$$\big( (r - s + 1)n + (r - s + 2)t \big) - t - \big( (r - s)n - (r - s + 1)t \big) = n,$$
which contradicts our assumption about the colouring. This proves the induction step, and hence the claim.
\end{proof}

By Claim~\ref{claim:4} with $s = r$, it follows that the colouring $c$ weakly contains a blue $K_{r+1}(t)$, as required.
\end{proof}

To complete the proof of Theorem~\ref{thm:general:precise}, we will need the following consequence of Janson's inequality, which we will also  need in Section~\ref{sec:klr1st}, below. We omit the (standard) proof.

\begin{lemma}\label{lemma:cliquespan}
Let $r \ge 2$ and $N = O(n)$, and let $p = x n^{-2/(r + 1)}$, where $x \gg 1$. With high probability, $G = G(N,p)$ has the following property: every collection $A_1,\ldots,A_{r+1} \subset V(G)$ of disjoint sets of size $\Omega\big( p^{-(r+1)/2} \log x \big)$ spans a copy of $K_{r + 1}$.
\end{lemma}

Now we use Proposition~\ref{prop:bluepartite} together with Lemma~\ref{lemma:cliquespan} to prove Theorem~\ref{thm:general:precise}.

\begin{proof}[Proof of Theorem \ref{thm:general:precise}]
Let $G = G(rn + t,p)$; we will show that if $G$ satisfies the conclusion of Lemma~\ref{lemma:cliquespan}, then (deterministically) $G \to (K_{r+1},P_n)$.
First, by Proposition~\ref{prop:bluepartite}, every red-blue colouring of the edges of $G$ either contains a red $P_n$, or weakly contains a blue $K_{r+1}(t')$, where $t'= \lfloor t/(r + 1) \rfloor$.
If we have a red $P_n$ then we are done, so let $A_1,\ldots,A_{r+1} \subset V(G)$ be disjoint sets with $|A_1| = \cdots = |A_{r+1}| = t'$ such that, for each $1 \le i < j \le r+1$, every edge of $G$ between $A_i$ and $A_j$ is coloured blue, and observe that $t' \ge p^{-(r+1)/2} \cdot \frac{\log x}{2(r+1)}$.

Now, by Lemma~\ref{lemma:cliquespan}, every such collection of sets spans a copy of $K_{r+1}$ and this copy of $K_{r+1}$ is coloured blue. It follows that, with high probability, every red-blue colouring of the edges of $G$ contains either a red copy of $P_n$ or a blue copy of $K_{r+1}$, as required.
\end{proof}

\section{The 1-statement in  the K\L R regime}\label{sec:klr1st}

In this section we will prove the following strengthening of Theorem~\ref{thm:klregime}. Let us fix $r \ge 2$ and set $C := 2^8r^2$ throughout this section.

\begin{thm}\label{thm:structural}
If $N \ge rn - n/3$ and $p \gg n^{-2/(r + 2)}$, then $G = G(N, p)$ has the following property with high probability as $n \to \infty$. For every red-blue colouring of the edges of $G$, at least one of following holds:
\begin{itemize}
\item[$(a)$] $G$ contains a blue copy of $K_{r + 1}$;
\item[$(b)$] $G$ contains a red copy of $P_n$;
\item[$(c)$] There exists a partition $V(G) = A_0 \cup A_1 \cup \cdots \cup A_r$, with $|A_0| \le C/p$ and $|A_i| \le n$ for each $i \in [r]$, such that every edge of $G[A_i,A_j]$ is blue for each $1 \le i < j \le r$.
\end{itemize}
\end{thm}

Note that Theorem~\ref{thm:klregime} follows immediately from Theorem~\ref{thm:structural}, since if $N > rn + C/p$ then property $(c)$ cannot hold, so with high probability every red-blue colouring of the edges of $G$ contains either a blue copy of $K_{r + 1}$ or a red copy of $P_n$.

In the proof of Theorem \ref{thm:structural}, we will use the minimum degree form of the sparse regularity lemma for colourings.
Given  $p \in (0,1)$ and $\eps > 0$, the $p$-density of a pair $(U,W)$ of disjoint sets of vertices in a graph $G$ is defined as
$$d_p(U,W) = \dfrac{e_G(U,W)}{p|U||W|}.$$
We say the pair $(U,W)$ is \emph{$(\eps,p)$-regular} in $G$ if $|d_p(U,W) - d_p(U',W')| \le \eps$ for all $U' \subset U$ and $W' \subset W$ with $|U'| \ge \eps |U|$ and $|W'| \ge \eps |W|$. Given $d > 0$, we say that the pair $(U,W)$ is \emph{$(\eps,d,p)$-regular} in $G$ if it is $(\eps,p)$-regular and also $d_p(U,W) \ge d$.

Now let us define an \emph{$(\eps,p)$-regular partition} for $2$-colourings.
Given a red-blue colouring of the edges of $G$, we write $G_R$ and $G_B$ for (respectively) the graphs on $V(G)$ induced by the red and blue edges.
We say that $V(G) = V_0 \cup V_1 \cup \ldots \cup V_k$ is an \emph{$(\eps, p)$-regular partition} for the colouring $(G_R,G_B)$ of $G$ if $|V_0| \le \eps n$ and $|V_1|= \cdots = |V_k|$, and moreover $(V_i, V_j$) is an $(\eps,p)$-regular pair in both $G_R$ and $G_B$ for all but at most $\eps k^2$ pairs $(i,j) \in [k]^2$.

Finally, the \emph{$(\eps,d,p)$-reduced graph} of an \emph{$(\eps, p)$-regular partition} $V = V_0 \cup V_1 \cup \cdots \cup V_k$ for a colouring $(G_R,G_B)$ of $G$ is the graph $R$ with vertex set $V(R) = \{1, \ldots, k\}$ and edge set
$$E(R) = \big\{ ij : (V_i,V_j) \text{ is } (\eps, p)\text{-regular in both } G_R \text{ and } G_B, \text{ and } d_p(V_i,V_j) \ge 2d \big\}.$$

We will use the following version of the sparse regularity lemma for random graphs.
This version follows easily from the coloured version~\cite[Theorem~7]{L16} and Chernoff's inequality.

\begin{lemma}[Sparse Regularity Lemma]\label{thm:regularity}
For each $\eps > 0$ and $k_0 \in \N$, there exists $k_1 \in \N$ such that for any $0 \le d < 1/2$, the following holds. If $p \gg (\log N)^4/N$, then with high probability every $2$-colouring of $G(N,p)$ has an $(\eps, p)$-regular partition into $k_0 \le k \le k_1$ parts, whose $(\eps,d,p)$-reduced graph has minimum degree at least $(1 - \eps)k$.
\end{lemma}

We will apply Lemma~\ref{thm:regularity} to our colouring of $G(N,p)$, colour the edges of $R$ with the denser colour, and apply the following lemma of Allen, Brightwell and Skokan~\cite[Lemma~19]{ABS}, which is a `stability version' of~\eqref{eq:Chvatal}.

\begin{lemma}\label{thm:stability}
Let $r \geq 2$ and $0 \leq \alpha \leq 1/2$, and suppose that $0 \leq \sigma \leq (1 - \alpha)/r$ and $K \geq 1/\sigma$. Let $G$ be a graph with $(r - \alpha)K$ vertices, and with minimum degree $\delta(G) \geq (r - \alpha - \sigma)K$. Then for every $2$-colouring of the edges of $G$, one of the following statements holds:
    \begin{itemize}
        \item[$(a)$] $G$ contains a blue copy $K_{r + 1}$;
        \item[$(b)$] $G$ contains a red copy of $P_K$;
        \item[$(c)$] $V(G)$ can be partitioned into $r$ parts, each of size at most $K$, such that every edge of $G$ within a part is red, and every edge between two different parts is blue.
    \end{itemize}
\end{lemma}

If $R$ contains a blue copy of $K_{r+1}$, then we will use the so-called ``K\L R conjecture", which was proved  by Balogh, Morris and Samotij~\cite{BMS}, Saxton and Thomason~\cite{ST} and Conlon, Gowers, Samotij, and Schacht~\cite{CGSS}, to find a blue $K_{r+1}$ in the colouring of $G(n,p)$. For simplicity, we will state this theorem only in the case we need, see~\cite[Theorem~1.6]{CGSS}.

\begin{thm}[The K\L R conjecture]\label{thm:klr}
For every $r \in \N$ and $d > 0$, there exists $\eps > 0$ such that if\/ $p \gg n^{-2/(r+2)}$, then $G = G(n,p)$ satisfies the following with high probability.

Let $V_1, \ldots, V_{r+1} \subset V(G)$ be disjoint sets of vertices with $|V_i| = \Omega(n)$ for each $i \in [r+1]$. If $(V_i, V_j)$ is $(\eps, d, p)$-regular for each $1 \le i < j \le r+1$, then $K_{r+1} \subset G$.
\end{thm}

If $R$ contains a long red path, then we will use the following lemma, which was proved by Letzter~\cite{L16}.

\begin{lemma}\label{lemma:reducedpath}
    Let $G$ be a graph, and let $V_1,\ldots,V_\ell \subset V(G)$ be disjoint sets of vertices. If $(V_i,V_{i+1})$ is $(\eps,d,p)$-regular for every $1 \le i < \ell$, then $G$ contains a cycle of length at least \[(1 - 8 \eps) \sum_{i = 1}^\ell |V_i|.\]
\end{lemma}

We are now ready to prove a weaker version of Theorem \ref{thm:structural}.

\begin{lemma}\label{lem:structural}
If $N \ge rn - n/3$ and $p \gg n^{-2/(r + 2)}$, then $G = G(N, p)$ has the following property with high probability as $n \to \infty$. For  every red-blue colouring of the edges of $G$, at least one of following holds:
\begin{itemize}
\item[$(a)$] $G$ contains a blue copy of $K_{r + 1}$;
\item[$(b)$] $G$ contains a red copy of $P_n$;
\item[$(c)$] $G$ contains $r$ vertex-disjoint red cycles, each of length greater than $n/2$.
\end{itemize}
\end{lemma}

\begin{proof}
Choose $d > 0$ sufficiently small, and let $\varepsilon = \varepsilon(d,r) > 0$ be given by Theorem \ref{thm:klr}.
Now set $k_0 := 1/\eps$, and let $k_1 = k_1(\eps,k_0) \in \N$ be given by Lemma~\ref{thm:regularity}. Since $n^{-2/(r + 2)} \gg (\log N)^4/N$, it follows by Lemma~\ref{thm:regularity} that, with high probability, every $2$-colouring of $G$ has an $(\eps, p)$-regular partition
$$V = V_0 \cup V_1 \cup \cdots \cup V_k$$
into $k_0 \le k \le k_1$ parts, whose $(\eps,d,p)$-reduced graph $R$ satisfies $\delta(R) \ge (1 - \eps)k$.

Let $(G_R,G_B)$ be a $2$-colouring of $G$, and colour each edge $ij$ of $R$ with the denser colour between $V_i$ and $V_j$.
Suppose first that $R$ contains a blue copy of $K_{r + 1}$, and note that (by the definition of $R$) each pair in the clique is $(\varepsilon, d, p)$-regular in $G_B$. It follows, by Lemma \ref{thm:klr}, that $G$ must contain a blue copy of $K_{r + 1}$.

Suppose next that there exists a red path in $R$ of length $K := (2r + 1)k / 2r^2$.
By Lemma~\ref{lemma:reducedpath}, it follows that $G$ contains a red path of length
$$(1 - 8\varepsilon) \frac{(K+1)(N - \eps n)}{k} \ge (1 - 9\varepsilon) \bigg( \frac{2r + 1}{2r^2} \bigg) \bigg( \frac{(3r - 1)n}{3} \bigg) > n,$$
since $\eps > 0$ was chosen sufficiently small.

Finally, suppose that $R$ contains neither a blue copy of $K_{r + 1}$, nor a red copy of $P_K$. In this case, we will apply Lemma~\ref{thm:stability} with
$$\alpha := \dfrac{r}{2r + 1} \qquad \text{and} \qquad \sigma = \frac{2\eps r^2}{2r + 1}.$$
Observe that $v(R) = k = (r - \alpha)K$, and that 
$$\delta(R) \ge (1 - \eps)k = (r - \alpha - \sigma)K$$
By Lemma \ref{thm:stability}, it follows that there exists a partition
$$V(R) = U_1 \cup U_2 \cup \cdots \cup U_r,$$
with $|U_i| \le K$ for each $i \in [r]$, such that all edges of $R$ inside an $U_i$ are red, and all edges of $R$ between two different sets $U_i$ and $U_j$ are blue.

We claim that, for each $i \in [r]$, the graph $R[U_i]$ contains a Hamiltonian cycle. To see this, simply observe that
$$\delta(R[U_i]) \geq |U_i| - \eps k > |U_i|/2,$$
where the first inequality follows from our lower bound on $\delta(R)$, and the second holds since $(\eps k \le \sigma K < (1 - \alpha)K/2 \le |U_i|/2$.
Hence, by Dirac's theorem, $R[U_i]$ contains a Hamiltonian cycle, as claimed.

By Lemma \ref{lemma:reducedpath}, and recalling that the edges of $R$ inside $U_i$ are all red, and that $|U_i| \ge (1-\alpha) K$, it follows that the set $B_i := \bigcup_{j \in U_i} V_j$ contains a red cycle in $G$ of length at least
$$(1 - 8\eps) |U_i| \cdot \frac{(N - \eps n)}{k} \geq (1 - 9\eps) \bigg( \frac{r+1}{2r+1} \bigg) \bigg( \frac{(2r + 1)k}{2r^2} \bigg) \bigg( \frac{(3r - 1)n}{3k} \bigg) > \frac{n}{2}.$$
Since the sets $B_1,\ldots,B_r$ are disjoint, this proves the lemma.
\end{proof}

Before deducing Theorem~\ref{thm:structural} from Lemma~\ref{lem:structural}, let us briefly sketch our strategy.
Let $B_1,\ldots,B_r$ be the (disjoint) vertex sets of $r$ red cycles of length greater than $n/2$ in $G$, and observe that if there exists a red edge between $B_i$ and $B_j$ for some $1 \le i < j \le r$, then $G$ contains a red copy of $P_n$.
Similarly, if any vertex not in $B_i \cup B_j$ has a red neighbour in both $B_i$ and $B_j$, then $G$ contains a red copy of $P_n$.
Moreover, if any vertex sends $\Omega(pn)$ blue edges to each $B_i$, then it follows by Lemma \ref{lemma:cliquespan} that $G$ contains a blue copy of $K_{r+1}$.

We can therefore (roughly speaking) define $A_i$ to be the set of vertices that send $\Omega(pn)$ edges into $B_i$, all but $o(pn)$ of which are red.
Noting (see Lemma~\ref{lemma:gnp}, below) that there are only $O(1/p)$ vertices with $o(pn)$ neighbours in $B_i$, it then only remains to show that $A_i$ contains a red Hamiltonian path in $G$ (and therefore has size at most $n$).
To do so, we will use the following lemma, essentially due to Pósa~\cite{P76} (see also~\cite[Lemma 8.6]{BB}), which provides a sufficient condition on the expansion of a graph for the existence of a long path.

\pagebreak

\begin{lemma}\label{lemma:expanderpath}
Let $k \in \N$, and let $G$ be a graph on $n$ vertices. If $|N(X)| \geq 2|X|$ for every set $X \subset V(G)$ with $|X| \leq k$, then there exists a path of length $\min\{3k - 1,n - 1\}$ in $G$.
\end{lemma}

We will use the following typical properties of the random graph $G(N,p)$, which may easily be proved using Chernoff's inequality.

\begin{lemma}\label{lemma:gnp}
If\/ $\alpha > 0$ is fixed and $p \gg N^{-2/(r + 2)}$, then the following hold with high probability for $G = G(N,p)$:
\begin{itemize}
\item[$(a)$] For every $U \subset V$ with $|U| \geq \alpha N$, there are at most $64/\alpha p$ vertices with at most $p|U|/8$ neighbors in $U$.
\item[$(b)$] $|N(X)| > (1 - 2e^{-C})N$ for every $X \subset V(G)$ with $|X| \ge C/p$.
\end{itemize}
\end{lemma}

We are now ready to prove Theorem~\ref{thm:structural}.

\begin{proof}[Proof of Theorem \ref{thm:structural}]
Let $N \ge rn - n/3$ and $p \gg n^{-2/(r + 2)}$, and suppose that $G = G(N, p)$ has the properties shown to hold with high probability in Lemmas~\ref{lemma:cliquespan},~\ref{lem:structural} and~\ref{lemma:gnp}. Suppose that we are given a red-blue colouring of $G$ that contains neither a red copy of $P_n$ nor a blue copy of $K_{r + 1}$. By Lemma~\ref{lem:structural}, it follows that $G$ contains $r$ vertex-disjoint red cycles, each of length greater than $n/2$. Let $B_1,\ldots,B_r$ be the vertex sets of these cycles, and observe that every edge of $G[B_i,B_j]$ is blue for every $1 \le i < j \le r$, since otherwise $G$ would contain a red copy of $P_n$. Fix $\alpha = 2^{-4}$, let $\gamma > 0$ be sufficiently small, and for each $i \in [r]$ define
$$A_i := \big\{ v \in V(G) : |N(v) \cap B_i| \ge \alpha p n \text{ and } |N_B(v) \cap B_i| \le \gamma p n \big\}.$$
Define $A_0 := V(G) \setminus \big( A_1 \cup \cdots \cup A_r \big)$. We first claim that $A_0$ is small.

\begin{claim}\label{claim:1}
$|A_0| \le C/p$.
\end{claim}

\begin{proof}[Proof of claim]
If $v \in A_0$ then either $|N(v) \cap B_i| \le \alpha p n$ for some $i \in [r]$, or $|N_B(v) \cap B_i| \ge \gamma p n$ for every $i \in [r]$. By Lemma~\ref{lemma:gnp}, and noting that $|B_i| \ge n/2 \ge N/4r$, there are at most $2^8r/p$ vertices with at most $p|B_i|/8$ neighbours in $B_i$. Since $p|B_i|/8 \ge pn/16 = \alpha pn$, it follows that there are at most $2^8r^2/p$ vertices with $|N(v) \cap B_i| \le \alpha p n$ for some $i \in [r]$.

On the other hand, if $|N_B(v) \cap B_i| \ge \gamma p n$ for every $i \in [r]$ then, recalling that all edges of each $G[B_i,B_j]$ are blue, and noting that
$$x := p (\gamma pn)^{2/r} \gg 1 \qquad \text{and} \qquad  \gamma p n \gg p^{-r/2} \log x,$$
it follows by Lemma~\ref{lemma:cliquespan} that $G$ contains a blue copy of $K_{r+1}$, which is a contradiction.
\end{proof}

Next, observe that the sets $A_1,\ldots,A_r$ are disjoint, since if there were a vertex with a red neighbour in $B_i$ and $B_j$ then $G$ would contain a red copy of $P_n$. Moreover, every edge of $G[A_i,A_j]$ is blue, since if there were a red path of length at most three from $B_i$ to $B_j$ then $G$ would contain a red copy of $P_n$.

It remains to show that $|A_i| \le n$ for each $i \in [r]$. To do so, we will use Lemma \ref{lemma:expanderpath} to prove the following claim.

\begin{claim}\label{claim:2}
$G[A_i]$ contains a red Hamiltonian path for each $i \in [r]$.
\end{claim}

\begin{proof}[Proof of claim]
By Lemma \ref{lemma:expanderpath}, it will suffice to show that $|N_R(X) \cap A_i| \ge 2|X|$ for every $X \subset A_i$ with $|X| \leq (|A_i| + 1)/3$. Observe first that the red degree in $A_i$ of each vertex $v \in A_i$ satisfies
$$|N_R(v) \cap A_i| \ge (\alpha - \gamma) p n - \frac{C}{p} \ge 2^{-5} p n,$$
by Claim~\ref{claim:1}, the definition of $A_i$, the observation that $v$ has no red neighbours in $B_j$ for any $j \ne i$, and the bound $p \gg n^{-1/2}$. The bound $|N_R(X) \cap A_i| \ge 2|X|$ follows immediately for every set $X \subset A_i$ with $|X| \le 2^{-6} p n$.

If $|X| \ge 2^{-6} p n$ then we will use Lemma~\ref{lemma:gnp} to show that $|N_R(X) \cap A_i| \ge 4|A_i|/5$. To see this, note that $|X| \ge C/p$, and let $Y \subset X$ with $|Y| = C/p$. By Lemma~\ref{lemma:gnp}, we have
$$|N_R(X) \cap A_i| \ge |N(Y) \cap A_i| - \gamma p n \cdot |Y| \ge |A_i| - 2e^{-C} N - C \gamma n \ge \frac{4|A_i|}{5},$$
as claimed, since $C = 2^8 r^2$ and $\gamma$ was chosen sufficiently small.
\end{proof}

Since we assumed that $G$ does not contains a red copy of $P_n$, it follows from Claim~\ref{claim:2} that $|A_i| \le n$ for each $i \in [r]$. Since we showed above that $A_0 \cup A_1 \cup \cdots \cup A_r$ is a partition of $V(G)$, that $|A_0| \le C/p$, and that every edge of $G[A_i,A_j]$ is blue for each $1 \le i < j \le r$, this completes the proof of Theorem \ref{thm:structural}.
\end{proof}

\section{Open Problems}\label{sec:openprob}

To conclude the paper, we will mention here a few natural directions for further research. First, it would be interesting to close the gap between the bounds given by Proposition~\ref{prop:specific} and Theorem~\ref{thm:general:precise}. Indeed, if $n^{-2/(r+1)} \ll p \ll n^{-2/(r+2)}$ and
$$p^{-\binom{r + 1}{2}}n^{-(r-1)} \le t \le p^{-(r+1)/2} \log\big( p n^{2/(r + 1)} \big),$$
then we do not know whether or not $G(rn + t,p) \to (K_{r+1}, P_n)$ with high probability.

\begin{prob}
For $n^{-2/(r+1)} \ll p \ll n^{-2/(r+2)}$, determine (up to a constant factor) the smallest $t$ for which $G(rn + t, p) \to (K_{r+1}, P_n)$ with high probability.
\end{prob}

It would also be interesting to determine a sharp threshold for $t$ (if one exists); in the range $p \gg n^{-2/(r+2)}$, it might even be possible to prove such a result via a more careful analysis of the method we used to prove Theorem~\ref{thm:klregime}.

Another natural direction would be to extend the results of this paper from cliques to arbitrary (fixed) graphs. It was proved in~\cite{EFRS} that for any graph $H$, the path $P_n$ is $H$-good for all sufficiently large $n$, and so it is natural to ask for the threshold of the event $G(N,p) \to (H,P_n)$ when $N = (1 + \eps)(\chi(H) -1)n$ for some fixed $\eps > 0$.

The construction used to prove Proposition~\ref{prop:naturalbarrier} can easily be generalised to show that if $p \ll n^{-1/m_1(H)}$, where $m_1(H) := \max\big\{ e(F)/(|F|-1) : F \subset H \big\}$, then $G(Cn,p) \not\to (H,P_n)$ with high probability, for any fixed $C > 0$, and it seems likely that this is the threshold.

\begin{prob}
Determine for which graphs $H$ we have $G((1+\eps)(\chi(H) -1)n, p) \to (H,P_n)$ with high probability for every fixed $\eps > 0$ and all $p \gg n^{-1/m_1(H)}$.
\end{prob}

Once again, it seems likely that the methods of this paper could be adapted to make progress on this problem, though some technical challenges remain (in particular, for unbalanced graphs $H$).

Finally, another natural extension of this work would be to replace the path $P_n$ by other $H$-good graphs on $n$ vertices (such as bounded-degree graphs with bandwidth $o(n)$, see~\cite{ABS}). We remark that, in forthcoming work with Ara\'ujo and Pavez-Sign\'e~\cite{AMP}, we have made progress on this problem in the case of a clique versus a bounded-degree tree.

\subsection*{Acknowledgements}
We would like to thank Robert Morris for fruitful discussions and suggestions throughout this work.

\end{document}